\documentclass{article}
\usepackage{graphicx} 
\usepackage[margin=1in]{geometry} 
\usepackage{amsmath,amsthm,amssymb,amscd}
\usepackage[alphabetic]{amsrefs}

\newcommand{\mf}[1]{\mathfrak{#1}}
\newcommand{\ZZ}{\mathbb{Z}}
\newcommand{\CC}{\mathbb{C}}

\DeclareMathOperator{\HF}{HF}
\DeclareMathOperator{\supp}{supp}
\DeclareMathOperator{\coker}{coker}
\DeclareMathOperator{\Ext}{Ext}
\DeclareMathOperator{\Tor}{Tor}
\DeclareMathOperator{\init}{in}
\DeclareMathOperator{\lcm}{lcm}
\DeclareMathOperator{\Ind}{Ind}

\DeclareMathOperator{\Sym}{Sym}
\DeclareMathOperator{\Span}{span}

\theoremstyle{definition}
\newtheorem{theorem}{Theorem}[section]
\newtheorem{proposition}[theorem]{Proposition}
\newtheorem{example}[theorem]{Example}

\title{Equivariant Syzygies of the Ideal of $2 \times 2$ Permanents of a $2 \times n$ Matrix}
\author{Jacob Zoromski}
\date{}

\begin{document}

\maketitle

\abstract{We describe the equivariant syzygies of the ideal of $2 \times 2$ permanents of a generic $2 \times n$ matrix under its natural symmetric and torus group actions. Our proof gives us a new method of finding the Betti numbers of this ideal, which were first described by Gesmundo, Huang, Schenck, and Weyman.}

\section{Introduction}

Let $V,W$ be complex vector spaces with bases $ \{ v_1,\dots,v_n \}$ and $\{w_1,\dots,w_m\}$, respectively. Let $S$ be the polynomial ring $\Sym(V \otimes W) \cong \CC[x_{11},\dots,x_{mn}]$, and $M$ be the $m \times n$ matrix such that $M_{ij} = x_{ij}$. Two natural ideals arising from $M$ are its ideal of subdeterminants and its ideal of subpermanents. Minimal free resolutions of the former were determined by Lascoux \cite{MR520233}, while resolutions of the latter remain mysterious. Our goal is to better understand the syzygies of the ideal of subpermanents.
\par
Recent progress was made in \cite{permanents}, where the authors find the graded Betti numbers of the ideal $P$ of the $2 \times 2$ subpermanents of $M$ when $m = 2$. Letting $x_{1i} = x_i$ and $x_{2j} = y_j$, $P$ can be written as
\[
P = (x_i y_j + x_j y_i : 1 \leq i < j \leq n) \subset \mathbb{C}[x_1,\dots,x_n,y_1,\dots,y_n]
\]
As a direction for future research, they proposed describing the syzygies of $P$ as representations over the symmetric groups that permute the rows and columns of $M$. The main theorem of our paper is a description of the syzygies of $P$ in terms of representations of a larger group which keeps track of multidegrees and can be described as follows.
\par
The symmetric group $\mf S_n \times \mf S_m$ acts on $V \otimes W$ by $(\sigma,\tau)(v_i \otimes w_j) = v_{\sigma(i)} \otimes w_{\tau(j)}$. We can also give $V \otimes W$ a $\ZZ^m \times \ZZ^n$ multigrading by letting $v_i, w_j$ have multidegrees $-\mathbf{e}_i, -\mathbf{f}_j$, where $\mathbf{e}_i,\mathbf{f}_j$ are standard basis elements of $\ZZ^m,\ZZ^n$, respectively. This corresponds to a $G_n \times G_m := (\mf S_n \ltimes (\CC^*)^n) \times (\mf S_m \ltimes (\CC^*)^m)$-action on $V \otimes W$, and hence on $S$. Our main theorem describes $\Ext^\bullet_S(P,\CC)$ in terms of $G_n \times G_2$-representations. For simplicity, we will only consider the $G_n$-action for the remainder of the introduction. To see the full $G_n \times G_2$ version of the main theorem, see Theorem \ref{P} of Section 3.
\par
The irreducible representations of $\mf S_n$ are the Specht modules, which are parametrized by partitions $\lambda = (\ell_1 \geq \ell_2 \geq \dots \geq \ell_q)$ with $|\lambda| =  \ell_1 + \dots + \ell_q = n$ and the $\ell_i$ positive integers. We will denote such representations by $[\lambda]$ (for background on complex representations of $\mf S_n$, see \cite[\S 4]{MR1153249}). Among the simplest representations to describe are the trivial representation $[n]$, the sign representation $[1^n]$, and, more generally, the Specht modules corresponding to hooks $[a,1^{n-a}]$ (here, $d^e$ means $d$ occurs $e$ times). These can be realized as exterior powers of the ``standard" representation $[n-1,1]$, so-called because the natural action of permuting the basis elements of an $n$-dimensional vector space gives the representation $[n] + [n-1,1]$. The representations occurring in our main theorem can be constructed entirely from these representations. 
\par
We will use $\Ind([\lambda_1],\dots,[\lambda_i])$ to denote the induced representation of $[\lambda_1] \boxtimes \dots \boxtimes [\lambda_i]$ from $\mf S_{|\lambda_1|} \times \dots \times \mf S_{|\lambda_i|}$ to $\mf S_{|\lambda_1| + \dots + |\lambda_i|}$. To recover the irreducible $\mf S_{|\lambda_1| + \dots + |\lambda_i|}$-representations from the induced representation, we use the Littlewood-Richardson rule, or the simpler Pieri's rule when, as in our main theorem, we induce up from trivial and sign representations (see \cite[\S A.1]{MR1153249} for background).
\par
 Representations of $G_n$ in the $\mf S_n$-orbit of a fixed multidegree $\mathbf{a}$ can be considered as induced representations from the $\mf S_n$-stabilizer of $\mathbf{a}$, which is a Young subgroup $\mf S_{p_1} \times \dots \times \mf S_{p_\ell}$ of $\mf S_n$. If $\mf S_{p_1} \times \dots \times \mf S_{p_\ell}$ acts on a vector space of multidegree $\mathbf a$ with representation $\rho_1 \boxtimes \dots \boxtimes \rho_\ell$, the $G_n$-representation corresponding to the $\mf S_n$-orbits of $\mathbf{a}$ is the induced representation 
\[
\Ind_{(\mf S_{p_1} \times \dots \times \mf S_{p_\ell}) \ltimes (\CC^*)^n}^{G_n}((\rho_1 \boxtimes \dots \boxtimes \rho_\ell) \otimes \chi_{\mathbf{a}})
\]
where $\chi_{\mathbf{a}}$ is the irreducible representation corresponding to multidegree $\mathbf{a}$. For simplicity, we will mostly consider $\mathbf{a}$ that are partitions, i.e. $\mathbf{a} = (a_1,\dots,a_n)$ with $a_1 \geq \dots \geq a_n$ and $a_i \geq 0$ for all $i$. When $\mathbf{a}$ is a partition, it can be rewritten as $(d_1^{p_1},\dots,d_\ell^{p_\ell})$ with $d_1 > \dots > d_\ell$. Abusing notation, we will write the above representation as 
\[
(\rho_1, \dots, \rho_\ell)_{\langle d_1^{p_1},\dots,d_\ell^{p_\ell}\rangle} \text{ or } (\rho_1, \dots, \rho_\ell)_{\langle \mathbf{a} \rangle}
\]
For a fixed $G_n$-invariant vector space $B$, the $G_n$-representation corresponding to the $\mf S_n$-orbits of $\mathbf{a}$ will be denoted by $B_{\langle \mathbf{a} \rangle}$. With these notations, we can state our main theorem:

\begin{theorem}
\label{Gn version}
The minimal free resolution of $P$ consists of three linear strands. The multidegrees $\mathbf{a}$ for which $\Ext_S^\bullet(P,\CC)_{\mathbf{a}}$ is non-zero are permutations of $(2^a,1^b,0^{n-a-b})$ for various $a$ and $b$. The $\mf S_n$-representations appearing in each strand can be determined from induced representations of trivial representations $[j]$, sign representations $[1^i]$, and the Specht module corresponding to the hook $[k-2,1^2]$ for various $i,j,$ and $k$. The following are isomorphisms of $G_n$-representations for each linear strand: 
\begin{itemize}
\item[First Strand:]
\[
\Ext^0(P,\CC) \cong ([2],[n-2])_{\langle 1^2,0^{n-2} \rangle}
\]
\item[Second Strand:] For $a,b$ with $2a + b = p+3, a \geq 1, b \geq 2$, and  $a + b \leq n$,
\[
\Ext^p(P,\CC)_{\langle 2^a,1^b,0^{n-a-b} \rangle} \cong
    ([a],\sum_{c = 0}^{b-2} \Ind([1^2],[1^{b-2-c}],[1^{c}]),[n-a-b])_{\langle 2^a,1^b,0^{n-a-b} \rangle} 
\]
\[
\Ext^p(P,\CC)_{\langle 1^{p+3},0^{n-p-3} \rangle} \cong
  (\sum_{c = 0}^{p+1}(\Ind([1^2],[1^{p+1-c}],[1^c]) - \Ind([1^{p+2-c}],[1^{c+1}]),[n-p-3])_{\langle 1^{p+3}, 0^{n-p-3} \rangle} \\
\]

\item[Third Strand:]
For $a,b$ with $2a + b = p+4, a \geq 3, b \geq 0,$ and $ a + b \leq n$, 
\[
\Ext^p(P,\CC)_{\langle 2^a,1^b,0^{n-a-b} \rangle} \cong
([a-2,1^2],\sum_{c = 0}^b \Ind([1^{b-c}],[1^{c}]),[n-a-b])_{\langle 2^a,1^b,0^{n-a-b} \rangle} 
\]

\end{itemize}

For all other $\mathbf{d}$, $\Ext(P,\CC)_{\langle \mathbf{d} \rangle} = 0$.
\end{theorem}

\begin{example}
When $n = 3$, the ideal $(x_1y_2 + x_2y_1,x_1y_3 + x_3y_1,x_2y_3+x_3y_2)$ is a complete intersection, whose resolution is given by the Koszul complex. Each of the three linear strands of the minimal free resolution is non-zero in the orbit of one multidegree:
\begin{align*}
\Ext^0(P,\CC) &\cong ([2],[1])_{\langle 1^2,0 \rangle} \\
\Ext^1(P,\CC) &\cong ([1],[1^2])_{\langle 2,1^2 \rangle} \\
\Ext^2(P,\CC) &\cong ([1^3])_{\langle 2^3 \rangle} 
\end{align*}

We can restrict to irreducible $\mf S_n$-representations by performing the induction in each multidegree using Pieri's rule. This results in the following isomorphisms of $\mf S_n$-representations:
\begin{align*}
\Ext^0(P,\CC)_2 &\cong [3] + [2,1] \\
\Ext^1(P,\CC)_4 &\cong [2,1] +[1^3] \\
\Ext^2(P,\CC)_6 &\cong [1^3]
\end{align*}
These are the exterior powers of $[3] + [2,1]$, as expected for the Koszul complex. When $n > 3$, $P$ is no longer a complete intersection, and its minimal free resolution is more complicated. For example, when $n=4$:
\begin{align*}
\Ext^0(P,\CC) &\cong ([2],[2])_{\langle 1^2,0^2 \rangle} \\
\Ext^1(P,\CC) &\cong ([1],[1^2],[1])_{\langle 2,1^2,0 \rangle} \oplus ([3,1] + 2[2,2] + [2,1^2])_{\langle 1^4 \rangle} \\
\Ext^2(P,\CC) &\cong ([1],2[2,1] + 2[1^3])_{\langle 2,1^3 \rangle} \oplus ([1^3],[1])_{\langle 2^3,0 \rangle} \\
\Ext^3(P,\CC) &\cong ([2],[1^2])_{\langle 2^2,1^2 \rangle} \oplus ([1^3],2[1])_{\langle 2^3,1 \rangle} \\
\Ext^4(P,\CC) &\cong ([2,1^2])_{\langle 2^4 \rangle} 
\end{align*}

\end{example}

\subsection{Graded Betti Numbers}
\par
In \cite{permanents}, the authors study the graded Betti numbers of $P$. They prove the following:
\begin{theorem}[\cite{permanents}, Thm. 1.2]
Let $\beta_{p,q} := \dim_\CC \Tor_p^S(P,\CC)_q$. Then
\[
\beta_{0,2} = \binom{n}{2}
\]
\[
\beta_{p,p+3} = 2 \binom{n}{p+3} - \binom{2n}{p+3} + \binom{n+1}{2} \binom{2n-2}{p+1} - 2 \binom{2n-3}{p} \binom{n}{2}
\]
\[
+ \sum_{a = 3}^{\lfloor \frac{p+3}{2} \rfloor} 2^{p + 3 - 2a} \binom{n}{a} \binom{n-a}{p+3-2a} \binom{a-1}{2}
\]
\[
\beta_{p,p+4} = \sum_{a = 3}^{\lfloor \frac{p+4}{2} \rfloor} 2^{p+4-2a} \binom{n}{a} \binom{n-a}{p+4-2a} \binom{a-1}{2}
\]
\[
\beta_{p,q} = 0 \text{ otherwise}
\]
with the convention that $\binom{a}{b} = 0$ pf $b < 0$ or $b > a$. 
\end{theorem}

We can easily recover a formula for the graded Betti numbers from Theorem \ref{Gn version}. For a representation $\rho_1 \boxtimes \dots \boxtimes \rho_\ell$ corresponding to a fixed multidegree $\mathbf a = (d_1^{p_1},\dots,d_\ell^{p_\ell})$, its dimension is $\prod_i \dim \rho_i$. The dimension of $(\rho_1,\dots,\rho_\ell)_{\langle \mathbf a \rangle}$, then, is $\prod \dim \rho_i$ times the number of $\mf S_n$-orbits of $\mathbf a$, namely $\binom{n}{p_1,\dots,p_\ell}\prod_i \rho_i$, where $\binom{n}{p_1,\dots,p_\ell}$ is the multinomial coefficient. Since $[i]$ and $[1^j]$ have dimension one, and $[a-2,1^2]$ has dimension $\binom{a-1}{2}$, we get the following formula for the graded Betti numbers:

\begin{theorem}
\label{our betti}
Let $\beta_{p,q} = \dim_\CC \Ext^p(P,\CC)_q = \dim_\CC \Tor^p(P,\CC)_q$. Then 

\[
\beta_{0,2} = \binom{n}{2} 
\]

\[
\beta_{p,p+3} = \sum_{a=0}^{\lfloor \frac{p+1}{2} \rfloor}  \binom{n}{a,p+3-2a,n+a-p-3} 2^{p+1-2a} - \binom{n}{p+3}(2^{p+3} - 2)
\]
\[
 = \binom{n}{2} \binom{2n-4}{p+1} - \binom{n}{p+3}(2^{p+3} - 2)
\]
\[
\beta_{p,p+4} = \sum_{a = 3}^{\lfloor \frac{p+4}{2} \rfloor}
 \binom{n}{a,p+4-2a,n+a-p-4}\binom{a-1}{2} 2^{p+4-2a}
\]
\[
= \binom{n}{2} \binom{2n-4}{p-1} - n \binom{2n-2}{p+1} + \binom{2n}{p+3} - \binom{n}{p+3} 2^{p+3}
\]
\end{theorem}
It is easy to see that the two formulas match in the third strand, though our formula for the second strand is a bit simpler. In fact, in the course of proving Theorem \ref{P}, we will easily arrive at the formulas without summation notation for the second and third linear strand.

To prove their formula, the authors of \cite{permanents} compute the third row of the Betti table of the initial ideal of $P$ using its Stanley-Reisner ring. Then, they use the Bernstein-Gelfand-Gelfand correspondence to show that the third row of the Betti table of $P$ is identical to that of its initial ideal. Finally, they use the Hilbert series of $P$ to determine the second row.
\par
Their proof doesn't give us access to the representation-theoretic aspects of the syzygies because it relied on the syzygies of the initial ideal, which is not invariant under symmetric group actions. Our proof, then, necessarily differs from theirs. We use short exact sequences coming from inclusions of $G_n \times G_2$-invariant ideals
\[
\tag{1}
0 \to (x_i y_j: i \neq j) \to (x_1,\dots,x_n)(y_1,\dots,y_n) \xrightarrow{\varphi} C_1 \to 0 
\]
\[
\tag{2}
0 \to P \to (x_i y_j: i \neq j) \xrightarrow{\psi} C_2 \to 0
\]
\par
The cokernels $C_1$ and $C_2$ are fairly simple, being direct sums of cyclic modules. We will apply $\Ext(-,\CC)$ to  (1) to describe $\Ext((x_i y_j : i \neq j),\CC)$, then use (2) to determine $\Ext(P,\CC)$. The morphisms involved have a lot of structure; they are morphisms both of $G_n \times G_2$-invariant vector spaces and of $\Ext(\CC,\CC) \cong \bigwedge(V \otimes W)^*$-modules. We take advantage of this structure to show that $\Ext(\varphi,\CC)$ and $\Ext(\psi,\CC)$ can be decomposed into injective, surjective, or zero maps in each multidegree, and that $\Ext((x_i y_j : i \neq j),\CC)$ and $\Ext(P,\CC)$ are direct sums of their kernels and cokernels. We then use the $G_n \times G_2$-representation structure of $\bigwedge(V \otimes W)^*$ to determine the representations of the kernels and cokernels.
\par

\subsection{Proof Outline}

\par
In this subsection, we will give an outline for the proof when $n=3$. Though $P$ is easily resolved by the Koszul complex in this case, the method of proof we use here applies with little modification when $n > 3$. The exact sequence (1) above in this case can be written as
\[
0 \to (x_1 y_2, x_2 y_1, x_1 y_3, x_3 y_1, x_2 y_3, x_3 y_2) \to (x_1,x_2,x_3)(y_1,y_2,y_3) \xrightarrow{\varphi} \substack{k[x_1,y_1] \overline{x_1 y_1} \\
\oplus k[x_2,y_2] \overline{x_2 y_2} \\
\oplus k[x_3,y_3] \overline{x_3 y_3} } \to 0
\]
Our first step is to apply $\Ext(-,\CC)$ to the sequence to determine $\Ext(D,\CC)$, where $D := (x_i y_j : i \neq j)$. This gives us the map $\widetilde \varphi := \Ext(\varphi,\CC)$ of $A := \bigwedge(V \otimes W)^*$-modules
\[
 \substack{A/(e_1,f_1) \overline{x_1 y_1} \\
\oplus A/(e_2,f_2) \overline{x_2 y_2} \\
\oplus A/(e_3,f_3) \overline{x_3 y_3} } \xrightarrow{\widetilde \varphi}(e_1,e_2,e_3)(f_1,f_2,f_3)
\]
where $e_i$ is dual to $x_i$ and $f_i$ is dual to $y_i$. The kernels and cokernels of $\widetilde \varphi$ determine $\Ext(D,\CC)$. The map $\widetilde \varphi$ sends $\overline{x_i y_i}$ to  $e_i f_i$. Letting $\mf S_3$ permute the indices of the basis elements $\overline{x_i y_i}$  makes $\widetilde \varphi$ a $G_3$-equivariant map. The ideal $(e_1,e_2,e_3)(f_1,f_2,f_3) \subset A$, as a multigraded vector space, lives in multidegrees $(2,0,0),(1,1,0),(2,1,0),(1,1,1),(2,2,0),(2,1,1),(2,2,1)$,and $(2,2,2)$, along with their $\mf S_3$-orbits, while the multidegrees appearing in the source of $\widetilde \varphi$ is the subset of those multidegrees which contain a 2. Now $\widetilde \varphi$ in degree $(2,0,0)$ sends $\overline{x_i y_i}$ to $e_i f_i$, so is surjective. Since $\widetilde \varphi$ is a $A$-module morphism, it is also surjective on all multidegrees containing a 2. But in each multidegree with no 2's, $\widetilde \varphi$ is the zero map. 
\par 
The kernel of $\widetilde \varphi$ is generated by $e_i f_i \overline{x_j y_j} - e_j f_j \overline{x_i y_i} $, and so lives in multidegrees $(2,2,0),(2,2,1),$ and $(2,2,2)$. The kernel and cokernel of $\widetilde \varphi$ live in distinct sets of multidegrees, so $\Ext(D,\CC)$ is the direct sum of the kernel and cokernel. The following table gives a basis for $\Ext(D,\CC)$ in each multidegree, along with its representation under the action of its stabilizing subgroup (e.g. multidegree $(1,1,0)$ has $\mf S_3$-stabilizer $\mf S_2 \times \mf S_1$, permuting the indices 1 and 2).

\begin{table}[h]
\centering
\resizebox{\columnwidth}{!}{
    \begin{tabular}{|c|c|c|c|}
    
        \hline 
        Module & Multidegree &  Basis & Representation\\ \hline \hline
        $\Ext^0(D,\CC)$ & $(1,1,0)$ & $e_1f_2,e_2f_1$ & $\Ind([1],[1]) \boxtimes [1]$ \\ \hline
        $\Ext^1(D,\CC)$&$(1,1,1)$ & $e_1 e_2 f_3, e_1 f_2 e_3, f_1 e_2 e_3, f_1 f_2 e_3, f_1 e_2 f_3, e_1 f_2 f_3$ &  $\Ind([1^2],[1]) + \Ind([1],[1^2])$ \\ \hline
         & $(2,2,0)$ & $e_1 f_1 \overline{x_2 y_2} - e_2 f_2 \overline{x_1 y_1}$ & $[1^2] \boxtimes [1]$\\ \hline
         $\Ext^2(D,\CC)$& $(2,2,1)$ & $e_3(e_1 f_1 \overline{x_2 y_2} - e_2 f_2 \overline{x_1 y_1}),f_3(e_1 f_1 \overline{x_2 y_2} - e_2 f_2 \overline{x_1 y_1})$ & $[1^2] \boxtimes [1] + [1^2] \boxtimes [1]$\\ \hline
        $\Ext^3(D,\CC)$& $(2,2,2)$ & $e_3 f_3(e_1 f_1 \overline{x_2 y_2} - e_2 f_2 \overline{x_1 y_1}),e_2 f_2(e_1 f_1 \overline{x_3 y_3} - e_3 f_3 \overline{x_1 y_1})$ & $[2,1]$\\ \hline
    \end{tabular}
    }
\end{table}

The representations column can be derived from the representations of $A$ and the properties of $\widetilde \varphi$. We see that $(e_1,e_2,e_3)(f_1,f_2,f_3)_\mathbf{a} \cong A_{\mathbf{a}}$ for any multidegree $\mathbf{a}$ in which $(e_1,e_2,e_3)(f_1,f_2,f_3)_{\mathbf{a}} \neq 0$. The representation $(\oplus A/(e_i,f_i) \overline{e_i f_i})_{\mathbf{a}}$ can be deduced by inducing from an exterior algebra in one fewer variable. For example, in multidegree $(2,2,0)$, $(\oplus A/(e_i,f_i) \overline{x_i y_i})$ has basis $e_1 f_1 \overline{x_2 y_2}, e_2 f_2 \overline{x_1 y_1}$, with representation $\Ind([1],[1]) \boxtimes [1]$, while $A$ has basis $e_1 f_1 e_2 f_2$ with representation $[2] \boxtimes [1]$. Since $\widetilde \varphi$ is surjective, we remove the latter from the former to get the correct representation of $\Ext^1(D,\CC)_{(2,2,0)} = [1^2] \boxtimes [1]$.
\par
Next, we apply $\Ext(-,\CC)$ to the short exact sequence (2) above, which becomes
\[
0 \to (x_1 y_2 + x_2 y_1,x_1 y_3 + x_3 y_1,x_2 y_3 + x_3 y_2) \to D \xrightarrow{\psi} \substack{S/(x_3,y_3,x_1 y_2 + x_2 y_1)\overline{x_1 y_2 - x_2 y_1} \\
\oplus S/(x_2,y_2,x_1 y_3 + x_3 y_1)\overline{x_1 y_3 - x_3 y_1} \\
\oplus S/(x_1,y_1,x_2 y_3 + x_3 y_2)\overline{x_2 y_3 - x_3 y_2} \\} \to 0
\]
This gives us the map $\tilde \psi = \Ext(\psi,\CC)$:
\[
\substack{A/(e_1,f_1,e_2,f_2) \overline{x_1 y_2 - x_2 y_1} \\
\oplus A/(e_1,f_1,e_3,f_3) \overline{x_1 y_3 - x_3 y_1} \\
\oplus A/(e_2,f_2,e_3,f_3) \overline{x_2 y_3 - x_2 y_3} \\} \oplus \substack{A/(e_1,f_1,e_2,f_2) \overline{x_1^2 y_2^2 - x_2^2 y_1^2} \\
\oplus A/(e_1,f_1,e_3,f_3) \overline{x_1^2 y_3^2 - x_3^2 y_1^2} \\
\oplus A/(e_2,f_2,e_3,f_3) \overline{x_2^2 y_3^2 - x_2^2 y_3^2} \\} \xrightarrow{\widetilde \psi} \Ext(D,\CC)
\]
It's easy to see that $\widetilde \psi(\overline{x_1 y_2 - x_2 y_1}) = e_1 f_2 - e_2 f_1$, and one can show that $\widetilde \psi(\overline{x_1^2 y_2^2 - x_2^2 y_1^2}) = e_1 f_1 \overline{x_2 y_2} - e_2 f_2 \overline{x_1 y_1}$, so $\widetilde \psi$ maps the two modules on the left to the separate degree components of $\Ext(D,\CC)$. So, $\Ext(D,\CC)$ again splits into a direct sum of kernels and cokernels, while in each multidegree, the representation of the source is an induced representation from an exterior algebra in fewer variables. In multidegree $(1,1,0)$, $\widetilde \psi$ is injective with cokernel $e_1 f_2 + e_2 f_1$ (corresponding to the generators of $P$) and representation $[2] \boxtimes [1]$, while in multidegree $(1,1,1)$, the map is surjective, as 
\[
\widetilde \psi ( \frac{1}{2} e_i \overline{x_j y_k - x_k y_j} - \frac{1}{2} e_j \overline{x_i y_k - x_k y_i} - \frac{1}{2} e_k \overline{x_i y_j - x_j y_i}) = e_i e_j f_k
\]
Indeed, it is an isomorphism in this multidegree. Since $\widetilde \psi$ maps onto the generators of $\ker \widetilde \varphi$, in multidegree $(2,2,0)$, it is surjective in multidegrees $(2,2,0),(2,2,1)$, and $(2,2,2)$. In fact, it is an isomorphism in multidegrees with two $2$'s. The following table summarizes the behavior of $\widetilde \psi$ in each multidegree.

\begin{table}[h]
\centering
    \begin{tabular}{|c|c|c|c|}
        \hline 
        Source of $\psi$ & Multidegree &  $\widetilde \psi$ & Representation of $\ker/\coker$ of $\psi$\\ \hline \hline
         & $(1,1,0)$ & Injective & $[2] \boxtimes [1]$ \\ \cline{2-4}
        $\oplus A/(e_i,f_i,e_j,f_j) \overline{x_i y_j - x_j y_i}$ &$(1,1,1)$ & Isomorphism &  0 \\ \cline{2-4}
        & $(2,1,1)$ & 0 & $[1] \boxtimes [1^2]$ \\ \hline
         & $(2,2,0)$ & Isomorphism & 0\\ \cline{2-4}
         $\oplus A/(e_i,f_i,e_j,f_j) \overline{x_i^2 y_j^2 - x_j^2 y_i^2}$ & $(2,2,1)$ & Isomorphism & 0\\ \cline{2-4}
        & $(2,2,2)$ & Surjection & $[1^3]$\\ \hline
    \end{tabular}
\end{table}

From this we can determine the representations appearing in $\Ext(P,\CC)$, namely,
\[
\Ext^0(P,\CC) \cong ([2],[1])_{\langle 1^2,0 \rangle}
\]
\[
\Ext^1(P,\CC) \cong ([1],[1^2])_{\langle 2,1^2 \rangle}
\]
\[
\Ext^2(P,\CC) \cong ([1^3])_{\langle 2^3 \rangle}
\]
The table above accurate describes the behavior of $\tilde \psi$ for all multidegrees of a similar form when $n > 3$. For example, on multidegrees $(2^a,1^b,0^{n-a-b})$, $\tilde \psi$ is 0 on the first source when $a = 1$ and $b \geq 2$, and on the second source $\tilde \psi$ is an isomorphism when $a = 2$. The only ``new" kernel or cokernel that appears is in multidegree $(1^b,0^{n-b})$ for $b > 3$, where $\tilde \psi$ is not an isomorphism but only surjective.

\subsection*{Acknowledgments}
I would like to thank my Ph.D advisor Claudiu Raicu for the idea for this project and for helpful guidance throughout.

\section{Representation Structure on the Exterior Algebra}
Our goal in this section is to describe the $G_n \times G_2$-representation structure of $A_\bullet = \Ext^\bullet(\CC,\CC) \cong \bigwedge^\bullet(V \otimes W)^*$, which is used often in the proof of our main theorem. To familiarize ourselves with $G_n$-representations, we first do an example.
\begin{example}
Suppose we want to find the $\mf S_n$-representation of the vector space consisting of all monomials of the form $x_i x_j, x_i y_j, y_i y_j$ in $S$, where $i \neq j$. These are all of the monomials in the $\mf S_n$-orbits of the vector space of multidegree $(1^2,0^{n-2})$, which is spanned by four monomials: $x_1 x_2, x_1 y_2, x_2 y_1$, and $y_1 y_2$. The stabilizer of this multidegree is $\mf S_2 \times \mf S_{n-2}$, where the permutation $(1,2)$ permutes the indices 1 and 2, while $\mf S_{n-2}$ acts trivially. The $\mf S_2$-action breaks this vector space down into four irreducible representations, spanned by $x_1 x_2$, $y_1 y_2$, $x_1y_2 + x_2y_1$, and $x_1y_2 - x_2y_1$, giving us $3 [2] \oplus  [1,1]$. Thus, the $G_n$-representation corresponding all monomials of the form $x_i y_j, x_i y_j, y_i y_j$ for $i \neq j$, is 
\[
S_{\langle 1^2,0^{n-2} \rangle} \cong 
(3 [2] + [1,1] , [n-2])_{\langle 1^2,0^{n-2}\rangle}
\]
\par
Since we have a $G_n \times G_2$-action, we can break this example down further. These monomials lie in the $\mf S_n \times \mf S_2$-orbit of two different $\ZZ^n \times \ZZ^2$-partitions: \\
$(1,1,0^{n-2}) \times (2,0)$, with vector space basis $ x_1 x_2$ \\
$(1,1,0^{n-2}) \times (1,1)$, with vector space basis $x_1 y_2 + x_2 y_1, x_1 y_2 - x_2 y_1$ \\
Permuting the indices 1,2 fixes $x_1 x_2$, and the $\mf S_2$-stabilizer of $(2,0)$ is the trivial group. So, the $G_n \times G_2$-representation corresponding to the $\mf S_n \times \mf S_2$-orbits of $(1,1,0^{n-2}) \times (2,0)$ in $S$, i.e. monomials of the form $x_i x_j, y_i y_j$, is 
\[
([2], [n-2])_{\langle 1^2,0^{n-2} \rangle} \boxtimes ([1],[1])_{\langle 2,0 \rangle}
\]
Whereas the vector space spanned by $x_1y_2 + x_2y_1$ is fixed both by permuting indices and swapping variables, while $x_1 y_2 - x_2 y_1$ is negated. So, the monomials $x_i y_j, x_j y_i$ correspond to the representation 
\[
([2] , [n-2])_{\langle 1^2,0^{n-2} \rangle} \boxtimes ([2])_{\langle 1,1 \rangle}
+ ([1,1] , [n-2])_{\langle 1^2,0^{n-2} \rangle} \boxtimes ([1,1])_{\langle 1,1 \rangle}
\]
\par
Restricting both of these representations to $G_n$ recovers the $G_n$-representation we found above.
\end{example}

We now begin our analysis of the $G_n \times G_2$-representations of $A = \bigwedge(V \otimes W)^*$. We let $e_1,\dots,e_n$, $
f_1,\dots,f_n$ be dual variables to $x_1,\dots,x_n,y_1,\dots,y_n$ generating $A$.
\par
The multidegrees that can appear in $A$ are the $\mf S_n \times \mf S_2$-orbits of $\mathbf{d} := (2^a,1^b,0^{n-a-b}) \times (a + b-c,a + c)$ for $a,b,c$ such that $a + b \leq n$ and $c \leq \lfloor \frac{b}{2} \rfloor$. $A$ in this multidegree has basis
\[
(e_1 f_1 \dots e_a f_a) e_{\sigma(a+1)} \dots e_{\sigma(a+b-c)} f_{\sigma(a+b-c+1)} \dots f_{\sigma(a + b)}
\]
for all permutations $\sigma \in \mf S_b$ that swap $\{a+1,\dots,a+b-c\}$ with $\{a+b-c+1,\dots,a+b\}$. Now the $\mf S_a$-action on $e_1 f_1 \dots e_a f_a$ is trivial, while the space of $e_{\sigma(a+1)} \dots e_{\sigma(a+c)} f_{\sigma(a+c+1)} \dots f_{\sigma(a + b)}$ can be viewed as an induced $\mf S_{b-c} \times \mf S_c$-representation by acting separately on $\wedge V^*$ and $\wedge W^*$, namely $\Ind([1^{b-c}], [1^{c}])$. So, as a $G_n$-representation,
\[
A_{\langle 2^a,1^b,0^{n-a-b} \rangle} \cong ([a],\Ind([1^{b-c}], [1^{c}]), [n-a-b])_{\langle 2^a, 1^b, 0^{n-a-b} \rangle}
\]
When $c \neq \frac{b}{2}$, $a + b - c \neq a+c$, so the $\mf S_2$ action on our basis is the induced trivial $\mf S_1 \times \mf S_1$ representation, i.e. the representation $[2] + [1^2]$. When $c = \frac{b}{2}$, we have to be careful, as the $\mf S_2$-representation will depend on the irrep appearing in 
\[
\Ind([1^{b-c}], [1^{c}]) = \sum_{j = 0}^c [2^{c - j},1^{2j}]
\]
To do this, we restrict the $GL_n(\CC) \times GL_2(\CC)$-representation of $A$  to its  $(2^a,1^{b},0^{n-a-b}) \times (a+c,a+c)$ weight space (see \cite[Chs.7,8]{MR1464693} for background). 
\[
\bigwedge(V \otimes W)^* \cong \bigoplus_{j = 0}^{c} \mathbb{S}_{(2^{a + c-j},1^{2j})} V^* \otimes \mathbb{S}_{(c+j,c-j)} W^*
\]
as $GL_n(\CC) \times GL_2(\CC)$-representations, where $\mathbb{S}_\lambda$ denotes the Schur functor corresponding to $\lambda$. The $(2^a,1^b,0^{n-a-b})$ weight space of $\mathbb{S}_{(2^{a+c-j},1^{2j})} V^*$ is simply $\Ind([a],[2^{c-j},1^{2j}])$. 
Now $\mathbb{S}_{(a+c+j,a+c-j)} W^* \cong ([\wedge^2 \mathbb{C}^2])^{\otimes a+c-j} \otimes \Sym^{2j}(\mathbb{C}^2)$.  The vector space $(\wedge^2 \mathbb{C}^2)^{\otimes a+c-j} \cong [1,1]^{\otimes a+c - j} $ has weight $(a+c-j,a+c-j)$, and the $(a+c,a+c)$ weight space of $\mathbb{S}_{(a+c+j,a+c-j)}$ is spanned by one vector, on which $\mf S_2$ acts trivially. To sum up,
\begin{proposition}
\label{A}
Let $\mathbf{d} = (2^a,1^b,0^{n-a-b})\times (a+b-c,a+c)$ with $a + b \leq n$ and $c \leq \lfloor \frac{b}{2} \rfloor$. The following is an isomorphism of $G_n \times G_2$-representations:
\[
A_{\langle \mathbf{d} \rangle} \cong 
\begin{cases}
([a],\Ind([1^{b-c}], [1^{c}]), [n-a-b])_{\langle 2^a, 1^b, 0^{n-a-b} \rangle} \boxtimes ([1],[1])_{\langle a+b-c,a+c \rangle} \text{ if }c \neq \frac{b}{2} \\
\sum_{j = 0}^c([a], [2^{c-j},1^{2j}],[n-a-b])_{\langle 2^a,1^{b},0^{n-a-b} \rangle }  \boxtimes ([1^2]^{a+c-j})_{\langle (a+c)^2 \rangle} \text{ if } c = \frac{b}{2}
\end{cases}
\]
\end{proposition}

\section{Proof of Main Theorem}

\subsection{Short Exact Sequences}

Recall the two short exact sequences we will use to prove our theorem:
\[
\tag{1}
0 \to D \to \mf{m}_{xy} \xrightarrow{\varphi} C_1 \to 0
\]
\[
\tag{2}
0 \to P \to D \xrightarrow{\psi} C_2 \to 0
\]
where $\mf{m}_{xy} = (x_1,\dots,x_n)(y_1,\dots,y_n)$, $D = (x_i y_j: i \neq j)$, and the maps on the left are inclusions. We will identify $C_1$ and $C_2$ as direct sums of cyclic modules.

\begin{proposition}
Let $S^{(i)} = S/(x_j: j \neq i)$ and $S^{(i,j)} = S/(x_k: k \neq i,j)$. Then
\[
C_1 \cong \bigoplus_{i} S^{(i)} \overline{x_i y_i} 
\]
\[
C_2 \cong \bigoplus_{i<j} \frac{S^{(i,j)}}{(x_i y_j + x_j y_i)} \overline{x_i y_j - x_j y_i}
\]
where $\overline{x_i y_i}$ is the image of $x_i y_i$ in $S^{(i)}$ and $\overline{x_i y_j - x_j y_i}$ is the image of $x_i y_j - x_j y_i$ in $S^{(i,j)}/(x_i y_j + x_j y_i)$. Under this identification, the maps $\varphi(x_i y_i) = \overline{x_i y_i}$  and, for $i <  j$, $\varphi(x_i y_j) = 0$ and $\psi(x_i y_j - x_j y_i) = \overline{x_i y_j - x_j y_i}$ make sequences (1) and (2) exact sequences and $\varphi$ and $\psi$ $G_n \times G_2$-equivariant maps.

\end{proposition}

The map $\varphi$ maps $S$-module generators of $\mf{m}_{xy}$ to generators of $C_1$, so it is surjective. It also clearly respects the group action. One set of minimal $S$-module generators for $D$ is $\bigcup_{i<j} \{x_i y_j + x_j y_i, x_i y_j - x_j y_i\}$, so again we can see that $\psi$ is surjective and $G_n \times G_2$ equivariant. 
\par 
To prove (1) and (2) are exact, we just need to show that the Hilbert functions of the modules appearing in (1) and (2), after making the proposed replacements, are additive. We will in fact use the $\ZZ^n$-graded Hilbert function for each module. To help with notation, we say that for $\mathbf{a} \in \ZZ^n$, $\supp \mathbf{a} = \{i : a_i \neq 0\}$.

\begin{proposition}
Let $\mathbf a = (a_1,\dots,a_n) \in \ZZ_{\geq 0}^n$ be a multidegree.

\[
\HF_{\mathbf{a}}(D) = \begin{cases}
          \prod_i(a_i + 1) - 2 \quad &\text{if } |\supp \mathbf a| \geq 2 \\
          0 \quad & \text{otherwise}
     \end{cases}
\]
\[
\HF_{\mathbf{a}}(\mf{m}_{xy}) = \begin{cases}
          \prod_i(a_i + 1) - 2 \quad &\text{if } |\supp \mathbf a| \geq 1 \\
          0 \quad & \text{otherwise}
     \end{cases}
\]
\[
\HF_{\mathbf{a}}(\bigoplus S^{(i)} \overline{x_i y_i}) = \begin{cases}
          a_i - 1 \quad &\text{if } \mathbf a = a_i \mathbf e_i \\
          0 \quad & \text{otherwise}
     \end{cases}
\]
\[
\HF_{\mathbf{a}}(P) = \begin{cases}
          \prod_i(a_i + 1) - 2 \quad &\text{if } |\supp \mathbf a| \geq 3 \\
          a_i a_j \quad &\text{if } \mathbf a = a_i \mathbf e_i + a_j \mathbf e_j \\
          0 \quad & \text{otherwise}
     \end{cases}
\]
\[
\HF_{\mathbf{a}}(\bigoplus \frac{S^{(i,j)}}{(x_i y_j + x_j y_i)} \overline{x_i y_j - x_j y_i}) = \begin{cases}
          a_i + a_j - 1 \quad &\text{if } \mathbf a = a_i \mathbf e_i + a_j \mathbf e_j \\
          0 \quad & \text{otherwise}
     \end{cases}
\]

\end{proposition}

\begin{proof}
    
Note that $\HF_{\mathbf{a}}(S) = \prod_i (a_i + 1)$. Now $\mf{m}_{xy}$ is generated by all monomials in $S$ of degree 2 that are a product of $x$'s and $y$'s, so the only monomials not in $(\mf{m}_{xy})_{\bf a}$ are $x_1^{a_1} \dots x_n^{a_n}$ and $y_1^{a_1} \dots y_n^{a_n}$. $D$ is generated by all monomials of degree 2 with support $\geq 2$ and are a product of $x$'s and $y$'s, so the only monomials not in $D_{\mathbf{a}}$ are again $x_1^{a_1} \dots x_n^{a_n}$ and $y_1^{a_1} \dots y_n^{a_n}$. A monomial in $S^{(i)} \overline{x_i y_i}$ of degree $d$ comes from multiplying $\overline{x_i y_i}$ by a monomial in $S^{(i)}$ of degree $d-2$, of which there are $d-1$. So, the first three Hilbert functions are done.
\par
The Hilbert function of $P$ is the same as that of its initial ideal $\init(P)$ under any monomial order. By \cite[Thm. 3.1]{MR1777172}, under the antidiagonal order, $\init(P)$ is generated by the monomials
\[
\begin{cases}
          x_i y_j\quad & i<j \\
          x_i y_j y_k \quad & i<j<k \\
          x_i x_j y_k \quad & i<j<k \\
     \end{cases}
\]
The only generators of $\init(P)$ with multidegree of support $2$ are $x_i y_j$ for $i \neq j$. So, an arbitrary monomial of multidegree $a_i \mathbf e_i + a_j \mathbf e_j$ is of the form $x_i^{\alpha_i} y_i^{a_i - \alpha_i} x_j^{\alpha_j} y_j^{a_j - \alpha_j}$ where $\alpha_i = 1,\dots,a_i$ and $\alpha_j = 0,\dots,a_j - 1$, for a total of $a_i a_j$. For a monomial with multidegree $\mathbf{a}$ of support at least three, consider three indices $i < j < k$ in the support. For such a monomial to not be in $\init(P)_{\mathbf a}$, either all $i,j,k$ correspond to $x$'s or all correspond to $y's$. So, the only monomials of multidegree $\mathbf{a}$ not in $\init(P)_{\mathbf{a}}$ are those with all $x$'s or $y$'s.
\par
In $\frac{S^{(i,j)}}{(x_i y_j + x_j y_i)}$, the ideal generated by $\overline{x_i y_j - x_j y_i}$ is the same as that generated by $2\overline{x_i y_j}$, a monomial. So, $\frac{S^{(i,j)}}{(x_i y_j + x_j y_i)} \overline{x_i y_j - x_j y_i}$ in multidegree $\mathbf a$ has a basis of equivalence classes of monomials of the form $x_i^{a_i} x_j^{\alpha_j} y_j^{a_j - \alpha_j}, x_i^{\alpha_i} y_i^{a_i - \alpha_i} y_j^{a_j}$ for $\alpha_j = 0,\dots,a_j$ and $\alpha_i = 1,\dots,a_i$, for a total of $a_i + a_j - 1$.
\end{proof}

\subsection{Equivariant Syzygies of $D$}

Now that we have our short exact sequences, we apply $\Ext(-,\CC)$ to (1) to determine $\Ext(D,\CC)$. For convenience, we will always let $\mathbf{d}$ denote a multidegree of the form $\mathbf{a} \times \mathbf{b} = (2^a,1^b,0^{n-a-b}) \times (a+b-c,a+c)$ for $a,b,c \geq 0$, $a + b \leq n$, and $c \leq \lfloor \frac{b}{2} \rfloor$. We will add assumptions on $a,b,$ and $c$ as needed. We will also omit the trivial representation on the $0$ portion of $\mathbf{a}$.
\begin{proposition}
\label{x_iy_i}
For $a \geq 1$, the following is an isomorphism of $G_n \times G_2$ representations:
\[
\Ext^{2a + b}(\oplus S^{(i)} \overline{x_i y_i},\CC)_{\langle \mathbf{d} \rangle} \cong 
\begin{cases}
(\Ind([1],[a-1]),\Ind([1^{b-c}], [1^{c}])_{\langle \mathbf{a} \rangle} \boxtimes ([1],[1])_{\langle {\mathbf{b}} \rangle} \text{ if }c \neq \frac{b}{2} \\
\sum_{j = 0}^c(\Ind([1],[a-1]), [2^{c-j},1^{2j}])_{\langle \mathbf{a} \rangle }  \boxtimes ([1^2]^{a-1+c-j})_{\langle \mathbf{b} \rangle} \text{ if } c = \frac{b}{2}
\end{cases}
\]
\end{proposition}

\begin{proof}
$S^{(i)} \overline{x_i y_i}$ is $G_{n-1} \times G_2$-invariant with $\mf S_{n-1}$ permuting all indices except $i$, and $\mf S_2$ acting trivially on the basis element $\overline{x_i y_i}$. $\Ext(S^{(i)} \overline{x_i y_i},\CC) \cong A^{(i)} \overline{x_i y_i}$, where $A^{(i)}$ is the $A$-module $A/(e_i,f_i)$ (we are abusing notation here, treating $\overline{x_i y_i}$ as just a $\CC$-vector with the same multigraded and symmetric group representations as $x_i y_i \in S$). Therefore, a basis element in $A^{(i)} \overline{x_i y_i}$ of multidegree $\mathbf{d}$ is of the form $\overline{x_i y_i}$ times a basis element of $A^{(i)}$ of multidegree $\mathbf{d} - (2\mathbf{e}_i) \times (1,1)$ so long as $i = 1,\dots,a$. To get the $G_n \times G_2$-representation of the direct sum  $\Ext(\oplus S^{(i)} \overline{x_i y_i},\CC) \cong \oplus A^{(i)} \overline{x_i y_i}$, then, we simply induce up on the $2^a$ portion of $\mathbf{d}$ from the corresponding $G_{n-1} \times G_2$-representations of $A^{(i)}$, which achieves the desired result.
\end{proof}

\begin{proposition}
\label{mef}
$\Ext^{2a+b}(\mf m_{xy},\CC)_{\langle \mathbf{d} \rangle} \cong (e_1,\dots,e_n)(f_1,\dots,f_n)_{\langle \mathbf{d} \rangle} \otimes_{\CC} ([n])_{(0^n)} \boxtimes ([1^2])_{(0^2)} =: (\mf m_{ef}^{-})_{\langle \mathbf{d} \rangle}$ as a submodule of the free $A$-module $A \otimes_{\CC} ([n])_{(0^n)} \boxtimes ([1^2])_{(0^2)} =: A^{-}$. In particular, for $\mathbf{d}$ with $2a + b \geq 2$ and if $a = 0$ then $c \geq 1$, we have the following isomorphism of $G_n \times G_2$-representations.
\[
\Ext^{2a + b}(\mf m_{xy},\CC)_{\langle \mathbf{d} \rangle} \cong \begin{cases}
([a],\Ind([1^{b-c}], [1^{c}]))_{\langle \mathbf{a} \rangle} \boxtimes ([1],[1])_{\langle \mathbf{b} \rangle} \text{ if }c \neq \frac{b}{2} \\
\sum_{j = 0}^c([a], [2^{c-j},1^{2j}])_{\langle \mathbf{a} \rangle }  \boxtimes ([1^2]^{a+1+c-j})_{\langle \mathbf{b} \rangle} \text{ if } c = \frac{b}{2}
\end{cases}
\]
\end{proposition}
\begin{proof}
First, as $\CC$-vector spaces, 
\begin{align*}
\Ext_S((x_1,\dots,x_n)(y_1,\dots,y_n),\CC) &\cong \Ext_{\CC[x_1,\dots,x_n]}((x_1,\dots,x_n),\CC) \otimes_{\CC} \Ext_{\CC[y_1,\dots,y_n]}((y_1,\dots,y_n),\CC) \\
 & \cong \bigwedge^{\geq 1} (V \otimes \Span_{\CC}\{w_1\})^* \otimes_{\mathbb{C}} \bigwedge^{\geq 1} (V \otimes \Span_{\CC}\{w_2\})^*
\end{align*}
 As an $A$-module, $\Ext(\mf m_{xy},\CC)$ is generated by $e_{i} \otimes f_{j}$ with $A$-module action
$e_i(e_j \otimes f_k) = e_i e_j \otimes f_k$ and $f_i(e_j \otimes f_k) = e_j \otimes f_i f_k$. The $\mf S_n$-action permutes indices while the $\mf S_2$-action swaps indices across the tensor, i.e. sends $e_{i} \otimes f_{j}$ to $e_{j} \otimes f_{i}$. $\Ext(\mf m_{xy},\CC)$ as an $A$ module is isomorphic to the ideal $(e_1,\dots,e_n)(f_1,\dots,f_n)$ of $A$ by sending $e_i \otimes f_j$ to $e_i f_j$. Tensoring $A$ with $([n])_{(0^n)} \boxtimes ([1^2])_{(0^2)}$ makes this map $G_n \times G_2$-equivariant, which we can check on the generators. The permutation swapping $e$'s and $f's$ acts on  $e_i \otimes f_i$ trivially, while it negates $e_i f_i$ in $A$. This permutation acting on $e_i \otimes f_j$ becomes $e_j \otimes f_i$, whose image in $A$ is $e_j f_i$, but when it acts on $e_i f_j$ it becomes $f_i e_j = -e_j f_i$. Adding the $[1^2]$ term for the $\mf S_2$-action fixes the sign.
\par
To get the representations in the conclusion, notice that for such $\mathbf{d}$, $(e_1,\dots,e_n)(f_1,\dots,f_n)_{\langle \mathbf{d} \rangle } \cong A_{\langle \mathbf{d} \rangle}$.
\end{proof}

Under the isomorphisms in the previous two propositions, $\widetilde \varphi = \Ext(\varphi,\CC)$, a $G_n \times G_2$-equivariant morphism of $A$-modules, sends $\overline{x_i y_i}$ to $e_i f_i$. Applying $\Ext(-,\CC)$ to (1) gives the long exact sequence
\[
\dots \to \oplus_i A^{(i)}_p \overline{x_i y_i} \xrightarrow{\widetilde{\varphi}^p}  (\mf{m}_{ef}^-)_p \to \Ext^p (D,\CC) \to \oplus_i A^{(i)}_{p+1} \overline{x_i y_i} \xrightarrow{\widetilde{\varphi}^{p+1}} \dots
\]

\begin{proposition}
\label{D}
    The following are isomorphisms of $G_n \times G_2$-representations:

    \begin{itemize}
    \item[(i)] For $\mathbf{d}$ with $a=0$, $b = p+2$ and $c \geq 1$,
    \[
    \Ext^p(D,\CC)_{\langle \mathbf{d} \rangle} \cong A_{\langle \mathbf{d} \rangle}^- \cong
    \begin{cases} (\Ind([1^{b-c}],[1^{c}])_{\langle \mathbf{a} \rangle} \boxtimes ([1],[1])_{\langle \mathbf{b} \rangle} \text{ if } c \neq \frac{b}{2}\\
    \sum_{j=0}^{c} ([2^{c - j},1^{2j}])_{\langle \mathbf{a} \rangle} \boxtimes ([1^2]^{c+1-j})_{\langle \mathbf{b} \rangle} \text{ if }c = \frac{b}{2}
    \end{cases}
    \]
    yielding graded Betti numbers
    \[
    \dim_\CC \Ext^p(D,\CC)_{p+2} = \binom{n}{p+2} (2^{p+2} - 2)
    \]
    \item[(ii)] For $\mathbf{d}$ with $a \geq 2$, $b \geq 0$, $2a + b = p+3$,
    \[
    \Ext^p(D,\CC)_{\langle \mathbf{d} \rangle} \cong (\oplus_j A^{(j)} \overline{x_j y_j})_{\langle \mathbf{d} \rangle } - A^-_{\langle \mathbf{d} \rangle} \cong
    \]
    \[
    \begin{cases}
        ([a-1,1],\Ind([1^{b-c}],[1^c]))_{\langle \mathbf{a} \rangle} \boxtimes ([1],[1])_{\langle \mathbf{b} \rangle} \text{ if } c \neq \frac{b}{2} \\
        \sum_{j=0}^{c} ([a-1,1],[2^{c - j},1^{2j}])_{\langle \mathbf{a} \rangle} \boxtimes ([1^2]^{a-1+c - j})_{\langle \mathbf{b} \rangle} \text{ if } c = \frac{b}{2}
    \end{cases}
    \]
    yielding graded Betti numbers
    \[
    \dim_\CC \Ext^p(D,\CC)_{p+3} = n \binom{2n-2}{p+1} + \binom{n}{p+3}2^{p+3} - \binom{2n}{p+3}
    \]
    \end{itemize}
    and $\Ext^p(D,\CC)_{\langle \mathbf{d} \rangle} = 0$ for all other $\mathbf{d}$.
\end{proposition}

\begin{proof}
  In multidegree $\mathbf{d}$ as in (i), the source of $\widetilde \varphi_{\mathbf{d}}^p$ is 0, while the target is non-zero, so its cokernel is isomorphic to $(\mf{m}^-_{ef})_{\mathbf{d}}  \cong A^-_{\mathbf{d}}$. The graded Betti numbers are then obtained by counting the squarefree monomials of degree $p+2$ in $A$ with at least one $e$ and $f$.
  \par 
  Multidegrees $\mathbf{d}$ with $a \geq 1, b \geq 0,$ and $2a + b = p+3$ are the only ones where source or target of $\widetilde \varphi^{p+1}_{\mathbf{d}}$ is non-zero. Now $\widetilde \varphi$ maps surjectively onto the $e_i f_i$, so it is surjective on all multidegrees with $a \geq 1$. Note for future reference that its kernel is generated by $e_j f_j \overline{x_i y_i} - e_i f_i \overline{x_j y_j}$ for $1 \leq i < j \leq a$. By examining the representations in the previous two propositions, we see that $(\oplus A^{(i)} \overline{x_i y_i})_{\langle \mathbf{d} \rangle} \cong A^-_{\langle \mathbf{d}\rangle }$ when $a=1$. The graded Betti numbers are obtained by counting all terms of degree $p+3$ in $\oplus A^{(i)} \overline{x_i y_i}$ and subtracting all non-squarefree monomials of degree $p+3$ of $A$.
  \par
  Since $\coker \widetilde \varphi^p$ and $\ker \widetilde \varphi^{p+1}$ are non-zero in distinct multidegrees, we must have
$\Ext^p(D,\CC) \cong \coker \widetilde \varphi^p \oplus \ker \widetilde \varphi^{p + 1}$. This concludes the proof.
\end{proof}

\subsection{Equivariant Syzygies of $P$}

Next, we apply $\Ext(-,\CC)$ to sequence (2). Here,
\[
\Ext^p(\frac{S^{(i,j)}}{(x_i y_j + x_j y_i)} \overline{x_i y_j - x_j y_i},\CC) \cong \Ext^p(S^{(i,j)} \overline{x_i y_j - x_j y_i},\CC) \oplus \Ext^{p-1}(S^{(i,j)} \overline{x_i^2 y_j^2 - x_j^2 y_i^2},\CC) 
\]
\[
\cong A_p^{(i,j)} \overline{x_i y_j - x_j y_i} \oplus A_{p-1}^{(i,j)} \overline{x_i^2 y_j^2 - x_j^2 y_i^2}
\]
where $A^{(i,j)} = A/(e_i,f_i,e_j,f_j)$ and, again, we consider $\overline{x_i y_j - x_j y_i}$ and $\overline{x_i^2 y_j^2 - x_j^2 y_i^2}$ to be vectors on which $\mf S_n \times \mf S_2$ acts as their corresponding elements do in $S$. 
\begin{proposition}
\label{det}
For $\mathbf{d}$ with $b \geq 2$, $c \geq 1$, we have an isomorphism of $G_n\times G_2$-representations
\[
(\oplus A^{(i,j)} \overline{x_i y_j - x_j y_i})_{\langle \mathbf{d} \rangle} \cong 
\begin{cases}
([a],\Ind([1^2],[1^{b-1-c}], [1^{c-1}]))_{\langle \mathbf{a} \rangle} \boxtimes ([1],[1])_{\langle \mathbf{b} \rangle} \text{ if }c \neq \frac{b}{2} \\
\sum_{j = 0}^{c-1}([a], \Ind([1^2],[2^{c-1-j},1^{2j}]))_{\langle \mathbf{a} \rangle }  \boxtimes ([1^2]^{a+c-j})_{\langle \mathbf{b} \rangle} \text{ if } c = \frac{b}{2}
\end{cases}
\]
For $\mathbf{d}$ with $a \geq 2$,
\[
(\oplus A^{(i,j)} \overline{x_i^2 y_j^2 - x_j^2 y_i^2})_{\langle \mathbf{d} \rangle} \cong 
\begin{cases}
(\Ind([1^2],[a-2]),\Ind([1^{b-c}], [1^{c}]))_{\langle \mathbf{a} \rangle} \boxtimes ([1],[1])_{\langle \mathbf{b} \rangle} \text{ if }c \neq \frac{b}{2} \\
\sum_{j = 0}^c(\Ind([1^2],[a-2]), [2^{c-j},1^{2j}])_{\langle \mathbf{a} \rangle }  \boxtimes ([1^2]^{a-1+c-j})_{\langle \mathbf{b} \rangle} \text{ if } c = \frac{b}{2}
\end{cases}
\]
\begin{proof}
$\overline{x_i y_j - x_j y_i}$ has representation $([1^2])_{\mathbf{e}_i + \mathbf{e}_j} \boxtimes ([1^2])_{(1,1)}$, and $\overline{x_i^2 y_j^2 - x_j^2 y_i^2}$ has representation $([1^2])_{2\mathbf{e}_i + 2\mathbf{e}_j}$ $ \boxtimes ([1^2])_{(2,2)}$, and the representations above are the induced $G_{n-2} \times G_2$ representations corresponding to the $\mathbf{d} - (\mathbf{e}_i + \mathbf{e}_j)\times (1,1)$ and $\mathbf{d} - (2\mathbf{e}_i + 2\mathbf{e}_j)\times (2,2)$ portions of $A^{(i,j)}$, respectively.
\end{proof}

\end{proposition}

Now $\widetilde{\psi} = \Ext(\psi,\CC)$ sends $\overline{x_i y_j - x_j y_i}$ to the image of $e_i f_j - e_j f_i$ in $\coker \widetilde \varphi^0 \cong (\mf m_{ef}^-)_2$. We need to show where $\widetilde \psi$ sends the other $A$-module generators, namely the $\overline{x_i^2 y_j^2 - x_j^2 y_i^2}$.

\begin{proposition}
\label{ext1}
Up to scaling, $\widetilde \psi(\overline{x_i^2 y_j^2 - x_j^2 y_i^2}) = e_i f_i \overline{x_j y_j} - e_j f_j \overline{x_i y_i} \in \ker \widetilde \varphi^2$.
\end{proposition} 

\begin{proof}

 Since $\widetilde \psi$ preserves multidegrees, the image of $\overline{x_i^2 y_j^2 - x_j^2 y_i^2}$ must be of degree $\mathbf{d} = (2\mathbf{e}_i + 2\mathbf{e}_j, (2,2))$. The only element in $\Ext(D,\CC)$ of this multidegree is $e_i f_i \overline{x_j y_j} - e_j f_j \overline{x_i y_i} \in \ker \widetilde \varphi$. Since $\mf S_n \times \mf S_2$ acts on these two elements identically (namely, both permuting $i$ and $j$ or swapping $e$'s with $f$'s and $x$'s with $y$'s both negate), the image of $\overline{x_i^2 y_j^2 - x_j^2 y_i^2}$ is either the proposed image or zero. 
\par
To show that $\widetilde{\psi}_{\mathbf{d}}$ is non-zero, we show explicitly that $\Ext^1(\psi,\CC)$ is non-zero in that multidegree. Consider $\frac{S^{(i,j)}}{(x_i y_j + x_j y_i)} \overline{x_i y_j - x_j y_i}$. It is a cyclic $S$-module with annihilator $(x_k,y_k | k \neq i,j) + (x_i y_j + x_j y_i)$, which is a complete intersection. The first step of its free resolution is then:
\[
 S b_{(\mathbf{e}_i + \mathbf{e}_j) \times (1,1)} \xleftarrow{\partial} S c_{\mathbf{d}}  \oplus \text{other multidegree terms}
\]
where $b_{(\mathbf{e}_i + \mathbf{e}_j)\times (1,1)}$ corresponds to $\overline{x_i y_j - x_j y_i}$ and $\partial (c_{\mathbf{d}}) = (x_i y_j + x_j y_i) b_{(\mathbf{e}_i + \mathbf{e}_j)\times (1,1)} $. Next, consider the first step of the minimal resolution of $D$. Since the generators of $D$ have pairwise distinct variables, the first step of the resolution is given by the Taylor complex:
\[
\bigoplus_{i \neq j} S X_i Y_j \xleftarrow{\partial'} 
\bigwedge^2 (\bigoplus_{i \neq j} S X_i Y_j )
\]
where $X_i Y_j$ corresponds to $x_i y_j$, and
\[
\partial' (X_i Y_j \wedge X_k Y_\ell) = \frac{\lcm(x_i y_j,x_k y_\ell)}{x_k y_\ell} X_k Y_\ell - \frac{\lcm(x_i y_j,x_k y_\ell)}{x_i y_j} X_i Y_j
\]
In particular, $\partial'(X_i Y_j \wedge X_j Y_i) = x_i y_j X_j Y_i - x_j y_i X_i Y_j$.
\par
The map $\psi: D \to D/P$ sending $x_i y_j$ to $\overline{x_i y_j - x_j y_i}$ induces the $S$-module map on the first homological degree of their projective resolutions $\oplus S X_i Y_j  \wedge X_k Y_\ell \to S c_{\mathbf{d}} \oplus \dots$ sending $X_i Y_j \wedge X_j Y_i$ to $c_{\mathbf{d}}$, so is non-zero in that multidegree. Thus, the induced map $\Ext^1(\psi,\CC)$ is also non-zero in multidegree $\mathbf{d}$. 
\end{proof}

\begin{theorem}
\label{P}
The following are isomorphisms of $G_n \times G_2$-representations:
\begin{itemize}
\item[(i)]
\[
\Ext^0(P,\CC)_{\langle (1^2,0^{n-2}) \times (1^2) \rangle } \cong ([2],[n-2])_{\langle 1^2,0^{n-2} \rangle} \boxtimes ([2])_{\langle 1^2 \rangle}
\]
    yielding graded Betti number
    \[
    \beta_{0,2} = \binom{n}{2}
    \]
\item[(ii)]
For $\mathbf{d}$ with $a=0$, $b = p+3$, $c \geq 1$, 
\[
\Ext^p(P,\CC)_{\langle \mathbf{d} \rangle} \cong (\oplus A^{(i,j)} \overline{x_i y_j - x_j y_i})_{\langle \mathbf{d} \rangle} - A^-_{\langle \mathbf{d} \rangle} \cong
\]
\[
\begin{cases}
(\Ind([1^2],[1^{b-1-c}],[1^{c-1}]) - \Ind([1^{b-c}], [1^{c}]))_{\langle \mathbf{a} \rangle} \boxtimes ([1],[1])_{\langle \mathbf{b} \rangle} \text{ if }c \neq \frac{b}{2} \\
\sum_{j = 0}^{c-1}(\Ind([1^2],[2^{c-1-j},1^{2j}]))_{\langle \mathbf{a} \rangle }  \boxtimes ([1^2]^{c-j})_{\langle \mathbf{b} \rangle} \\
- \sum_{j = 0}^c([2^{c-j},1^{2j}])_{\langle \mathbf{a} \rangle }  \boxtimes ([1^2]^{1+c-j})_{\langle \mathbf{b} \rangle} \text{ if } c = \frac{b}{2}
\end{cases}
\]
\item[(iii)]
For $\mathbf{d}$ with $2a + b = p+3, a \geq 1$, and $b \geq 2$,
\[
\Ext^p(P,\CC)_{\langle \mathbf{d} \rangle} \cong (\oplus A^{(i,j)} \overline{x_i y_j - x_j y_i})_{\langle \mathbf{d} \rangle} \cong
\]
\[
\begin{cases}
([a],\Ind([1^2],[1^{b-1-c}], [1^{c-1}]))_{\langle \mathbf{a} \rangle} \boxtimes ([1],[1])_{\langle \mathbf{b} \rangle} \text{ if }c \neq \frac{b}{2} \\
\sum_{j = 0}^{c-1}([a], \Ind([1^2],[2^{c-1-j},1^{2j}]))_{\langle \mathbf{a} \rangle }  \boxtimes ([1^2]^{  a+c-j})_{\langle \mathbf{a} \rangle} \text{ if } c = \frac{b}{2}
\end{cases}
\]
yielding graded Betti numbers
    \[
    \beta_{p,p+3} = \binom{n}{2} \binom{2n-4}{p+1} - \binom{n}{p+3}(2^{p+3} - 2)
    \]
\item[(iv)]
For $\mathbf{d}$ with $2a + b = p+4, a \geq 3$, and $b \geq 0$, 
\[
\Ext^p(P,\CC)_{\langle \mathbf{d} \rangle} \cong (\oplus A^{(i,j)}\overline{x_i^2 y_j^2 - x_j^2 y_i^2})_{\langle \mathbf{d} \rangle} - \ker \widetilde \varphi_{\langle \mathbf{d} \rangle} \cong
\]
\[
\begin{cases}
([a-2,1^2],\Ind([1^{b-c}], [1^{c}]))_{\langle \mathbf{a} \rangle} \boxtimes ([1],[1])_{\langle \mathbf{b} \rangle} \text{ if }c \neq \frac{b}{2} \\
\sum_{j = 0}^c([a-2,1^2], [2^{c-j},1^{2j}])_{\langle \mathbf{a} \rangle }  \boxtimes ([1^2]^{ a-1+c-j})_{\langle \mathbf{b} \rangle} \text{ if } c = \frac{b}{2}
\end{cases}
\]
yielding graded Betti numbers
    \[
    \beta_{p,p+4} = \binom{n}{2} \binom{2n-4}{p-1} - n \binom{2n-2}{p+1} + \binom{2n}{p+3} - \binom{n}{p+3} 2^{p+3}
    \]

\end{itemize}

For other $\mathbf{d}$, $\Ext(P,\CC)_{\langle \mathbf{d} \rangle} = 0$.

\end{theorem}

\begin{proof}

Applying $\Ext(-,\CC)$ to (2) gives long exact sequence
\[   
 \dots \to \substack{\oplus A^{(i,j)}_p \overline{x_i y_j - x_j y_i} \\ \oplus A^{(i,j)}_{p - 1}\overline{x_i^2 y_j^2 - x_j^2 y_i^2} }  \xrightarrow{\widetilde \psi^p} \substack{ \coker \widetilde \varphi^{p}   \\ \oplus \ker \widetilde \varphi^{p+1} } \to \Ext^p(P,\CC) \to \dots 
\]
By our description of $\widetilde \psi$, it can be decomposed into 
\[
\alpha^p: \oplus A_p^{(i,j)} \overline{x_i y_j - x_j y_i} \to \coker \widetilde \varphi^p
\]
\[
\beta^p: \oplus A^{(i,j)}_{p-1}\overline{x_i^2 y_j^2 - x_j^2 y_i^2} \to  \ker \widetilde \varphi^{p+1}
\]
The map $\alpha^0$ is injective; $\overline{x_1y_2 - x_2 y_1}$ has representation 
$([1^2])_{(1^2)} \boxtimes ([1^2])_{(1^2)}$
and removing this representation from $(\coker \widetilde \varphi)_{((1^2,0^{n-2}),(1^2))} \cong A^-_{(1^2)\times (1^2)}$ leaves the representation appearing in (i).
\par
I claim that $\alpha^p$ is surjective for $p > 0$. It suffices to show that $\alpha^1$ is surjective, for which it suffices to show that any element $e_i e_j f_k$ of $A^-$ for $i < j < k$ is in the image of $\alpha^1$. Indeed,
\[
\alpha( \frac{1}{2} e_i \overline{x_j y_k - x_k y_j} - \frac{1}{2} e_j \overline{x_i y_k - x_k y_i} - \frac{1}{2} e_k \overline{x_i y_j - x_j y_i}) = e_i e_j f_k
\]
So, for $\mathbf{d}$ as in (ii), the source of $\alpha_{\mathbf{d}}$ is $(\oplus A^{(i,j)} \overline{x_i y_j - x_j y_i})_{\mathbf{d}}$ and its target is isomorphic to $A^-_{\mathbf{d}}$, so the representation of $\ker \alpha_{\mathbf{d}}$ is the difference between their representations. This leaves the representations appearing in (ii) (when $a = 0$ and $b = 3$, the two representations are identical, so the kernel is 0). 
\par
When $a \geq 1$, the target of $\alpha_{\mathbf{d}}$ is zero, so the kernel's representation is $(\oplus A^{(i,j)} \overline{x_i y_j - x_j y_i})_{\mathbf{d}}$, which gives the representations of (iii). To obtain the graded Betti numbers, note that we include the entire degree $p+3$ component of $\oplus A^{(i,j)} \overline{x_i y_j - x_j y_i}$, and we exclude only the squarefree monomials of $A$ of degree $p+3$ with at least one $e$ or $f$. This is all of the contributions of $\alpha$ to $\Ext(P,\CC)$. 
\par
The map $\beta$ is surjective, as it maps onto the generators of $\ker \widetilde \varphi$. So, its representation in degree $\mathbf{d}$ is the difference between that of $(\oplus A^{(i,j)} \overline{x_i^2 y_j^2 - x_j^2 y_i^2})_\mathbf{d}$ and $\ker \widetilde \varphi_{\mathbf{d}}$. Since $\Ind([1^2],[a-2]) = [a-2,1^2] + [a-1,1]$, the difference of representations in the multidegree 2 part of the $\mathbb{Z}^n$ portion is $[a-2,1^2]$ if $a \geq 3$ and 0 if $a = 2$. This gives the representations appearing in $(iv)$. To obtain the graded Betti numbers, note that we include the entire degree $p+3$ component of $\oplus A^{(i,j)} \overline{x_i^2 y_j^2 - x_j^2 y_i^2}$, and exclude the entire $p+3$ component of $\ker \tilde \varphi$, which has the same formula as the Betti numbers in part (ii) of Proposition \ref{D}.
\par
Finally, $\alpha$ and $\beta$ are maps on separate linear components, and their kernels and cokernels lie in distinct multidegrees. So, $\Ext(P,k)$ is the direct sum of the various kernels and cokernels of $\alpha$ and $\beta$, and we are done.

\end{proof}

\bibliography{mybibliography}

\end{document}